\documentclass[a4paper,11pt]{article}
\usepackage{amsmath, amsfonts, amscd, amssymb, amsthm, enumerate, xypic}

\def\bfB{\mathbf{B}}

\newcommand{\id}{\operatorname{id}}

\newcommand{\End}{\operatorname{End}}
\newcommand{\Vect}{\operatorname{span}}
\newcommand{\im}{\operatorname{Im}}

\newcommand{\tr}{\operatorname{tr}}

\renewcommand{\setminus}{\smallsetminus}

\def\F{\mathbb{F}}
\def\K{\mathbb{K}}

\def\N{\mathbb{N}}

\def\lcro{\mathopen{[\![}}
\def\rcro{\mathclose{]\!]}}

\theoremstyle{definition}
\newtheorem{Def}{Definition}

\theoremstyle{plain}
\newtheorem{theo}{Theorem}
\newtheorem{prop}[theo]{Proposition}
\newtheorem{cor}[theo]{Corollary}
\newtheorem{lemma}[theo]{Lemma}

\theoremstyle{plain}

\theoremstyle{remark}
\newtheorem{Rems}{Remarks}
\newtheorem{Rem}[Rems]{Remark}

\title{Sums of quadratic endomorphisms of an infinite-dimensional vector space}
\author{Cl\'ement de Seguins Pazzis\footnote{Universit\'e de Versailles Saint-Quentin-en-Yvelines, Laboratoire de Math\'ematiques
de Versailles, 45 avenue des Etats-Unis, 78035 Versailles cedex, France}
\footnote{e-mail address: dsp.prof@gmail.com}}

\begin{document}

\thispagestyle{plain}

\maketitle

\begin{abstract}
We prove that every endomorphism of an infinite-dimensional vector space over a field splits into the sum of four idempotents and into the sum of four square-zero endomorphisms, a result that is optimal in general.
\end{abstract}

\vskip 2mm
\noindent
\emph{AMS Classification:} 15A24; 16B50

\vskip 2mm
\noindent
\emph{Keywords:} Infinite dimension, Decomposition, Square-zero endomorphism, Idempotent, Ordinals.

\section{Introduction}

In trying to decompose an endomorphism of a vector space into a sum/linear combination/product of endomorphisms
of special type, two situations are traditionally studied:
\begin{itemize}
\item The one of
finite-dimensional vector spaces, i.e.\ the matrix case, see e.g. \cite{BallantineI,BallantineII,BallantineIII,BallantineIV,Erdos,WangWu};
\item In the infinite-dimensional setting one considers a real or complex Hilbert space
(or Banach space) and the endomorphisms and the summands (or factors) are required to be bounded operators (see e.g.
\cite{Dawlings,Fillmore,PearcyTopping,WangWu,Wu}).
\end{itemize}

In this article, we shall explore a somewhat neglected territory, in which the vector space is assumed to be infinite-dimensional
and the ground field is totally arbitrary. Hence, there is no structure from analysis involved here and the problem is a purely algebraic one.
Here are basic questions: Which endomorphisms can be written as a (finite) sum of idempotents? of involutions? of square-zero endomorphisms?
If so, what is the minimal number of summands required in such a decomposition?
In those questions, the special endomorphisms are all \textbf{quadratic}, a quadratic endomorphism $u$
being one that satisfies $u^2 \in \Vect(\id,u)$. Hence, we will consider more generally  the question of decomposing
an arbitrary endomorphism into the sum of quadratic endomorphisms with prescribed split annihilating polynomials
of degree $2$.

Before we go on, let us introduce some notation. Throughout the text, $\F$ is an arbitrary field, and $t$
is an indeterminate which we use to write polynomials over $\F$.
We use the French convention for the set of all non-negative integers, which we denote by $\N$.
All the vector spaces that we consider have $\F$ as ground field.
An endomorphism $u$ of an $\F$-vector space $V$ endows $V$ with a structure of $\F[t]$-module
so that $t.x=u(x)$ for all $x \in V$: We use the notation $V^u$ when we speak of $V$ as an $\F[t]$-module.
The endomorphism $u$ is called \textbf{elementary} when the module $V^u$
is free, i.e.\ when it is isomorphic to the direct sum of (potentially infinitely many) copies of $\F[t]$.
 Our basic method will be to start from an endomorphism $u$ of $V$ and, by subtracting well-chosen ``special" endomorphisms,
to obtain an elementary one. In this prospect, the key notion is the one of a \emph{stratification} of an $\F[t]$-module,
which we will define and study in Section \ref{stratsection}.

\section{Main results}

We start with the main results of our article.

\begin{Def}
Let $p_1,\dots,p_n$ be polynomials with coefficients in $\F$.
An endomorphism $u$ of a vector space $V$ is called a \textbf{$(p_1,\dots,p_n)$-sum}
when there exists an $n$-tuple $(u_1,\dots,u_n)$ of endomorphisms of $V$ such that
$$u=\sum_{k=1}^n u_k \quad \text{and} \quad \forall k \in \lcro 1,n\rcro, \; p_k(u_k)=0.$$
\end{Def}

\begin{theo}\label{theo2}
Let $p_1$ and $p_2$ be split polynomials of degree $2$ with coefficients in $\F$.
Then, every elementary endomorphism of a vector space is a $(p_1,p_2)$-sum.
\end{theo}

\begin{theo}\label{theo4}
Let $p_1,p_2,p_3,p_4$ be split polynomials of degree $2$ with coefficients in $\F$.
Then, every endomorphism of an infinite-dimensional vector space is a $(p_1,p_2,p_3,p_4)$-sum.
\end{theo}

\begin{cor}\label{4cor}
Every endomorphism of an infinite-dimensional vector space is the sum of four square-zero endomorphisms
and the sum of four idempotent endomorphisms.
\end{cor}

R. S\l{}owik recently showed \cite{Slowik} that over a field with characteristic different from $2$,
any endomorphism of a vector space with (infinite) countable dimension is the sum of ten square-zero endomorphisms.
On the other hand, in an infinite-dimensional complex Hilbert space it is known that every bounded operator is the sum of five square-zero ones,
and the result is optimal \cite{PearcyTopping}.

By checking the details of the proof of Theorem \ref{theo4}, the reader will easily convince himself
that Corollary \ref{4cor} can be extended to left vector spaces over an arbitrary division ring.
Better still, Theorems \ref{theo2} and \ref{theo4} can be generalized to that setting provided that
the polynomials under consideration have all their coefficients central in the said division ring
(the proof is similar but cannot be framed in the theory of modules).

Before we go on, let us discuss the optimality of Corollary \ref{4cor}:

\begin{prop}\label{not3squarezero}
Let $u$ be a finite-rank endomorphism of a vector space $V$, and assume that $u$ is the sum of
three square-zero endomorphisms of $V$. Then, $\tr(u)=0$.
\end{prop}

In particular, no rank $1$ idempotent of $V$ is the sum of three square-zero endomorphisms since its trace equals $1$,
which proves that Corollary \ref{4cor} is optimal as far as square-zero endomorphisms are concerned.

\begin{proof}[Proof of Proposition \ref{not3squarezero}]
Let us split $u=a+b+c$ for square-zero endomorphisms $a,b,c$ of $V$.
We claim that the finite-dimensional subspace
$$W:=\im u+\im (au)+\im (bu)+\im (cu)+\im (abu)+\im (acu)+\im (bcu)+\im (abcu)$$
is stable under $a$, $b$ and $c$. It is obvious that $W$ is stable under $a$ since $a^2=0$.
Next, we claim that $W$ contains
$$W':=\im u+\im (au)+\im (bu)+\im (cu)+\im (bau)+\im (acu)+\im (bcu)+\im (bacu).$$
Indeed, we note first that $ab+ba=(a+b)^2=(u-c)^2=u(u-c)-cu$, and hence $\im(bau) \subset \im(abu)+\im(u)+\im(cu) \subset W$.
Likewise, we find $\im (bacu) \subset \im(abcu)+\im(u)+\im(cu) \subset W$, which proves the claimed inclusion.
Symmetrically, one obtains $W' \subset W$, and hence $W'=W$. As $W'$ is stable under $b$ -- again, because $b^2=0$ --
we conclude that $W$ is stable under $b$. Finally, $W$ contains $\im u$ and hence it is stable under $u$,
whence $W$ is stable under $c=u-a-b$.

From there, we denote by $u',a',b',c'$ the endomorphisms of $W$ induced by $u,a,b,c$ respectively, so that $u'=a'+b'+c'$.
Note now that $\tr(u')=\tr(u)$ since $W$ contains $\im u$: indeed, by completing a basis $\bfB$ of $\im u$ into a basis $\bfB'$ of $W$,
we get that the matrix of $u'$ in $\bfB'$ reads
$$\begin{bmatrix}
A & [?] \\
[0] & [0]
\end{bmatrix}$$
where $A$ denotes the matrix of $u_{|\im u}$ in $\bfB$, and hence $\tr u'=\tr A=\tr u_{|\im u}=\tr u$.

 On the other hand $a'$, $b'$ and $c'$ are square-zero endomorphisms of $W$,
 which leads to $\tr(a')=\tr(b')=\tr(c')=0$. Hence, $\tr(u)=\tr(u')=0$.
\end{proof}

Now, let us turn to idempotents:

\begin{prop}\label{not3idem}
Let $\alpha \in \F \setminus \bigl\{0,1,2,3\bigr\}$ be such that $2\alpha \neq 3$, and let $V$ be a vector space over $\F$.
Then, $\alpha\,\id_V$ is not the sum of three idempotent endomorphisms of $V$.
\end{prop}

\begin{proof}
Assume on the contrary that $\alpha\,\id_V=p+q+r$ for some idempotents $p,q,r$.
Note first that $q$ and $r$ both commute with $(q+r)(2\id_V-q-r)=q+r-qr-rq$
(for instance, $q(q+r-qr-rq)=q-qrq=(q+r-qr-rq)q$).
However, $(q+r)(2\id_V-q-r)=(\alpha \id_V-p)((2-\alpha)\id_V+p)=\alpha(2-\alpha) \id_V+(2\alpha-3) p$.
Since $2\alpha-3 \neq 0$, we deduce that both $q$ and $r$ commute with $p$.
Symmetrically, $q$ commutes with $r$. Hence, $p,q,r$ are simultaneously diagonalizable,
which leads to $\alpha$ being the sum of three elements of $\{0_\F,1_\F\}$, contradicting the assumption that $\alpha \not\in \{0,1,2,3\}$.
\end{proof}

In particular, if $\F$ has characteristic $2$ and more than $2$ elements,
no finite-rank endomorphism of $V$ with trace outside of the prime subfield of $\F$
is the sum of three idempotents. We suspect however that
every endomorphism of $V$ is the sum of three idempotents if $\F=\F_2$:
indeed, the result is known to hold over finite-dimensional spaces, see \cite{dSPidempotentLC}.

\begin{Rem}\label{directsumremark}
Let $u$ be an endomorphism of a vector space $V$. Let $p_1,\dots,p_r$ be polynomials with coefficients in $\F$.
Assume that $V$ splits as $V=\underset{i \in I}{\bigoplus} V_i$ where each linear subspace $V_i$ is stable under $u$
and we denote by $u_i$ the induced endomorphism. Assume also that for all $i \in I$,
the endomorphism $u_i$ splits as $u_i=\underset{k=1}{\overset{r}{\sum}} u_{i,k}$ where $u_{i,k}\in \End(V_i)$ and $p_k(u_{i,k})=0$
for all $k \in \lcro 1,r\rcro$.
Then, by setting $u^{(k)}:=\underset{ i \in I}{\bigoplus} u_{i,k}$, we see that
$u=\underset{k=1}{\overset{r}{\sum}} u^{(k)}$ and $p_k(u^{(k)})=0$ for all $k \in \lcro 1,r\rcro$.
\end{Rem}

\begin{Rem}[The canonical situation]\label{canonicalremark}
In both Theorems \ref{theo2} and \ref{theo4}, we can reduce the situation to the one where
each polynomial under consideration has the form $t^2-at$ for some $a \in \F$.
Indeed, let $p_1,\dots,p_r$ be split polynomials with degree $2$ over $\K$.
For each $k$, denote by $x_k,y_k$ the roots of $p_k$ and note that an endomorphism $v$ is annihilated by $p_k$
if and only if $v-x_k\id$ is annihilated by $t^2-(y_k-x_k)t$.
Given $u \in \End(V)$, we deduce that $u$ is a $(p_1,\dots,p_k)$-sum
if and only $u-\Bigl(\underset{k=1}{\overset{r}{\sum}} x_k\Bigr)\,\id$ is a $\bigl(t^2-(y_1-x_1)t,\dots,t^2-(y_r-x_r)t\bigr)$-sum.
In both Theorems \ref{theo2} and \ref{theo4}, we note that the assumption on the endomorphism $u$
is left invariant by subtracting a scalar multiple of the identity from $u$ (for Theorem \ref{theo2}, note that $V^u$ is free if and only if $V^{u-\lambda \id}$ if free, owing to the fact that $p(t) \mapsto p(t+\lambda)$ is an automorphism of the $\F$-algebra $\F[t]$).
Hence, in both theorems, it will suffice to consider the case when each polynomial $p_k$
has the form $t^2-at$ for some $a \in \F$ (depending on $k$).
\end{Rem}

Theorem \ref{theo2} will be proved in Section \ref{elementarysection}.
The proof of Theorem \ref{theo4} is spread over Sections \ref{stratsection} and \ref{theo4proofsection}.

\section{Decomposing an elementary operator}\label{elementarysection}

Here, we give a quick proof of Theorem \ref{theo2}.
We will need the following basic lemma, which is folklore:

\begin{lemma}\label{elementarylemma}
Let $u$ be an endomorphism of a vector space $V$ with (infinite) countable dimension, and let $(e_n)_{n \in \N}$ be a basis of $V$.
Assume that $u(e_n)=e_{n+1}$ mod $\Vect(e_0,\dots,e_n)$ for all $n \in \N$. Then,
$(u^n(e_0))_{n \in \N}$ is a basis of $V$.
\end{lemma}

In light of Remark \ref{canonicalremark}, we can limit the discussion to the following situation:
Let $u$ be an elementary endomorphism of a vector space $V$, and $a$ and $b$ be scalars.
We have to prove that there exist endomorphisms $v$ and $w$ of $V$ such that $u=v+w$,
$v^2=a\,v$ and $w^2=b\,w$.
In this prospect, we know from Remark \ref{directsumremark} that it suffices to prove this when $V^u$ has
a single generator, say $x$. Since all the non-zero free $\F[t]$-modules with one generator are isomorphic,
we can turn the question on its head: It suffices to
\emph{construct} a non-zero vector space
$U$ over $\F$ and a pair $(v,w)$ of endomorphisms of $U$ such that $v^2=a\,v$, $w^2=b\,w$, and the $\F[t]$-module $U^{v+w}$ has a generator.
Indeed, assume that we have done so and let $V$ be a non-zero vector space and $u$ be an endomorphism of $V$ 
such that $V^u$ is free with one generator; then, we choose an isomorphism
$\varphi : V^u \rightarrow U^{v+w}$, and we see that $v':=\varphi^{-1} \circ v \circ \varphi$ and $w':=\varphi^{-1} \circ w \circ \varphi$
are endomorphisms of $V$ that are respectively annihilated by $t^2-at$ and $t^2-bt$, while $u=v'+w'$. 

Now, we take an arbitrary vector space $U$ over $\F$ with a countable basis $(e_n)_{n \in \N}$
on which we define two endomorphisms $v$ and $w$ as follows:
\begin{itemize}
\item $v(e_k)=a\, e_k+e_{k+1}$ for every even $k \in \N$, otherwise $v(e_k)=0$;
\item $w(e_k)=b\, e_k+e_{k+1}$ for every odd $k\in \N$, otherwise $w(e_k)=0$.
\end{itemize}
For every $k \in \N$, we have $v^2(e_k)=0=a\,v(e_k)$ if $k$ is odd, otherwise
$v^2(e_k)=v(a\,e_k+e_{k+1})=a\,v(e_k)+0=a\, v(e_k)$ since $k+1$ is even.
Hence, $v^2=a\, v$. Likewise, one proves that $w^2=b\, w$.

Setting $u:=v+w$, we see that $u$ satisfies the condition of Lemma \ref{elementarylemma}, whence
$(u^n(e_0))_{n \in \N}$ is a basis of $U$. It follows that $U^u$ is a free $\F[t]$-module with one generator, which
completes the proof of Theorem \ref{theo2}.

\section{Stratifications}\label{stratsection}

\subsection{Stratifications and associated objects}

\begin{Def}
Let $V$ be an $\F[t]$-module.
A \textbf{stratification} of $V$ is an increasing sequence $(V_\alpha)_{\alpha \in D}$, indexed over a well-ordered set $D$,
of submodules of $V$ in which:
\begin{itemize}
\item For all $\alpha \in D$, the quotient module $V_\alpha/\biggl(\underset{\beta<\alpha}{\sum} V_\beta\biggr)$
is non-zero and has a generator;
\item $V=\underset{\alpha \in D}{\sum} V_\alpha$.
\end{itemize}
To any such stratification, we assign the \textbf{dimension sequence} $(n_\alpha)_{\alpha \in D}$ defined by $$n_\alpha:=\dim_\F\biggl(V_\alpha/\underset{\beta<\alpha}{\sum} V_\beta\biggr)$$
(in the infinite-dimensional case, we consider
the dimension to be $+\infty$, not the first infinite ordinal $\aleph_0$).
\end{Def}

Let $(V_\alpha)_{\alpha \in D}$ be a stratification of $V$.
For every $\alpha \in D$, we can choose a vector $x_\alpha \in V_\alpha$ such that $V_\alpha=\F[t] x_\alpha +\underset{\beta<\alpha}{\sum} V_\beta$,
and we note that if $n_\alpha$ is finite then
$V_\alpha=\F_{n_\alpha-1}[t]\,x_\alpha \oplus \underset{\beta<\alpha}{\sum} V_\beta$ and $(t^k x_\alpha)_{0 \leq k<n_\alpha}$
is linearly independent, otherwise $(t^k x_\alpha)_{0 \leq k<+\infty}$ is linearly independent and
$V_\alpha=\F[t]\,x_\alpha \oplus \underset{\beta<\alpha}{\sum} V_\beta$.
We shall say that the \textbf{vector sequence} $(x_\alpha)_{\alpha \in D}$ is attached to $(V_\alpha)_{\alpha \in D}$.
In this case, an obvious transfinite induction shows that, for all $\alpha$ and $\beta$ in $D$ with $\beta<\alpha$,
the family $(t^k\, x_\delta)_{\beta \leq \delta \leq \alpha, \; 0 \leq k <n_\delta}$ is linearly independent and
\begin{equation}\label{complet}
\begin{cases}
\biggl(\underset{\gamma<\beta}{\sum} V_\gamma\biggr) \oplus \Vect\bigl((t^k x_\delta)_{\beta \leq \delta \leq \alpha, \; 0 \leq k <n_\delta}\bigr)
& =V_\alpha \\
\biggl(\underset{\gamma<\beta}{\sum} V_\gamma\biggr) \oplus \Vect\bigl((t^k x_\delta)_{\beta \leq \delta < \alpha, \; 0 \leq k <n_\delta}\bigr)
& =\underset{\gamma<\alpha}{\sum} V_{\gamma.}
\end{cases}
\end{equation}
In particular, $(t^k\,x_\alpha)_{\alpha \in D, \; 0 \leq k <n_\alpha}$ is a basis of $V$.
As a special case, we get the obvious consequence:

\begin{lemma}\label{allinfinity}
Let $V$ be an $\F[t]$-module with a stratification $(V_\alpha)_{\alpha \in D}$.
Assume that the corresponding dimension sequence $(n_\alpha)_{\alpha \in D}$ is constant with sole value $+\infty$.
Then, $V$ is free.
\end{lemma}

Conversely, consider a sequence $(x_\alpha)_{\alpha \in D}$, indexed over a well-ordered set $D$,
of vectors of $V$ such that $x_\alpha \not\in \underset{\beta <\alpha}{\sum} \F[t]\, x_\beta$ for all
$\alpha \in D$, and $V=\underset{\alpha \in D}{\sum} \F[t]\, x_\alpha$.
Then, one sees that $\Bigl(\underset{\beta \leq \alpha}{\sum} \F[t]\,x_\beta\Bigr)_{\alpha \in D}$
is a stratification of $V$ with corresponding vector sequence $(x_\alpha)_{\alpha \in D}$.

Of course, any stratification can be re-indexed over an ordinal, and we will often
assume that the stratifications we are dealing with are indexed over ordinals.

\subsection{Connectors for a stratification}

\begin{Def}\label{defconnector}
Let $u$ be an endomorphism of a vector space $V$.
Let $(V_\alpha)_{\alpha \in D}$ be a stratification of $V^u$, with attached dimension sequence
$(n_\alpha)_{\alpha\in D}$ and an associated vector sequence $(x_\alpha)_{\alpha \in D}$.

An endomorphism $v$ of $V$ is called a \textbf{connector} for $u$ with respect to the vector sequence
$(x_\alpha)_{\alpha \in D}$ whenever it acts as follows on the basis $(t^k x_\alpha)_{\alpha \in D, 0 \leq k<n_\alpha}$:
For all $\alpha \in D$ such that $n_\alpha<+\infty$, we have $v(t^{n_{\alpha}-1}\,x_\alpha)=x_{\alpha+1}$ mod $V_\alpha$,
and all the other vectors are mapped to $0$.
\end{Def}

This definition is motivated by the following result:

\begin{prop}\label{connector}
Let $u$ be an endomorphism of a vector space $V$.
Let $(V_\alpha)_{\alpha \in \kappa}$ be a stratification of $V^u$, with attached dimension sequence
$(n_\alpha)_{\alpha\in \kappa}$ and an associated vector sequence $(x_\alpha)_{\alpha \in \kappa}$.

Assume that if $\kappa$ has a maximum $M$ then $n_M=+\infty$. Then, for any connector
$v$ for $u$ with respect to $(x_\alpha)_{\alpha \in \kappa}$, the endomorphism $u+v$ is elementary.
\end{prop}

\begin{proof}
Without loss of generality, we can assume that $\kappa$ is an ordinal.

We define $D$ as the set of all $\alpha \in \kappa$ such that
either $\alpha$ has no predecessor or $\alpha$ has a predecessor and $n_{\alpha-1}=+\infty$.
Note that $D$ is a non-empty well-ordered set.

Set $w:=u+v$.

Fix $\alpha \in D$. Either $\alpha+k \in \kappa$ and $n_{\alpha+k}<+\infty$ for all $k \in \N$, in which case we set $m_\alpha:=+\infty$, otherwise we denote by $m_\alpha$ the smallest positive integer $k$ such that $\alpha+k-1 \in \kappa$
and $n_{\alpha+k-1}=+\infty$ (it exists because if $\kappa$ has a maximum $M$ then $n_M=+\infty$).
In any case, we set
$$W_\alpha:=\sum_{ k<m_\alpha} V_{\alpha+k}=\underset{ k<m_\alpha}{\bigcup} V_{\alpha+k.}$$
We shall prove that $(W_\alpha)_{\alpha \in D}$ is a stratification of $V^w$
and that the corresponding dimension sequence takes no finite value. Lemma \ref{allinfinity} will then imply that $w$ is elementary.

To help us, we need additional notation. Let $\beta \in \kappa$.
If $\beta \not\in D$ then $\beta$ has a predecessor $\beta-1$.
As there is no infinite decreasing sequence in $\kappa$,
it follows that there is a uniquely-defined element $g(\beta)\in D$ together with an non-negative integer $m$ such that
$g(\beta)+m=\beta$ and $g(\beta)+k \not\in D$ for all $k \in \lcro 1,m\rcro$.
It follows from the above definition
that
$$V_\beta \subset W_{g(\beta)}.$$

Now, for all $\beta \in \kappa$, the endomorphism $v$ maps $V_\beta$ into $V_{\beta+1}$
unless $n_\beta=+\infty$ in which case $V_\beta$ is stable under $v$.
Hence, we deduce from the definition of $m_\alpha$ that $W_\alpha$ is stable under $v$.
Moreover, we readily see that it is stable under $u$, and we conclude that $W_\alpha$ is a submodule of $V^w$.

Next, we see that $(W_\alpha)_{\alpha \in D}$ is increasing.
Let indeed $(\alpha,\beta)\in D^2$ be such that $\alpha<\beta$.
From the definition of $m_\alpha$, it follows that $\alpha+k<\beta$ for every integer $k$
such that $0 \leq k<m_\alpha$, and hence $W_\alpha \subset \underset{\gamma \in \kappa, \gamma<\beta}{\sum} V_\gamma
\subsetneq V_\beta \subset W_\beta$.

Finally, let us fix $\alpha \in D$.
Denote by $y$ the class of $x_\alpha$ in the quotient module
$$E:=W_\alpha/\underset{\beta \in D, \; \beta<\alpha}{\sum} W_\beta.$$
Let us prove that $E$ is non-zero and is the free module generated by $y$.
We start by proving that
\begin{equation}
\label{eq1}
\underset{\beta \in D, \; \beta<\alpha}{\sum} W_\beta=\underset{\beta \in \kappa, \; \beta<\alpha}{\sum} V_\beta.
\end{equation}
Let $\beta \in \kappa$ be such that $\beta<\alpha$.
Then, $V_\beta \subset W_{g(\beta)}$ with $g(\beta)<\alpha$ and $g(\beta)\in D$.
Conversely, let $\beta \in D$ be such that $\beta<\alpha$.
Then, $\beta+k<\alpha$ for all $k$ such that $0 \leq k<m_\beta$, and hence it follows from the definition of $W_\beta$ that
$W_\beta \subset \underset{\gamma \in \kappa, \; \gamma<\alpha}{\sum} V_\gamma$.

Using equality \eqref{eq1}, we deduce that
$$E=\biggl(\sum_{0 \leq k<m_\alpha} V_{\alpha+k}\biggr)/\biggl(\underset{\beta \in \kappa, \; \beta<\alpha}{\sum} V_\beta\biggr).$$
Assume first that $m_\alpha=+\infty$. Then, $n_{\alpha+k}$ is finite for each $k\in \N$
and we set
\begin{multline*}
(e_n)_{n \in \N}=\bigl(x_\alpha,u(x_\alpha),\dots,u^{n_\alpha-1} (x_\alpha),x_{\alpha+1},u(x_{\alpha+1}),\dots,
u^{n_{\alpha+1}-1} (x_{\alpha+1}),\dots,\\
x_{\alpha+k},u(x_{\alpha+k}),\dots,
u^{n_{\alpha+k}-1} (x_{\alpha+k}),\dots\bigr).
\end{multline*}
If $m_\alpha$ is finite, we set
\begin{multline*}
(e_n)_{n \in \N}=\bigl(x_\alpha,u(x_\alpha),\dots,u^{n_\alpha-1} (x_\alpha),x_{\alpha+1},u(x_{\alpha+1}),\dots, u^{n_{\alpha+1}-1}(x_{\alpha+1}),\dots,\\
x_{\alpha+m_\alpha-1},u(x_{\alpha+m_\alpha-1}),\dots,u^l(x_{\alpha+m_\alpha-1}),\dots\bigr)
\end{multline*}
In any case we know from \eqref{complet} that, for any integer $k$ such that $0 \leq k<m_\alpha$, the vectors
$x_\alpha,\dots,u^{n_{\alpha+k}-1}(x_{\alpha+k})$ are linearly independent and
$$\Vect\bigl(x_\alpha,\dots,u^{n_{\alpha+k}-1}(x_{\alpha+k})\bigr)\oplus \biggl(\underset{\beta \in \kappa, \; \beta<\alpha}{\sum} V_\beta\biggr)
=V_{\alpha+k.}$$

Next, we prove that $w(e_n)=e_{n+1}$ mod $\underset{\beta \in \kappa, \; \beta<\alpha}{\sum} V_\beta+
\Vect(e_0,\dots,e_n)$ for all $n \in \N$.
Let indeed $n \in \N$. Then, $e_n=t^l x_{\alpha+k}$ for some $0 \leq k<m_\alpha$ and some $0 \leq l<n_{\alpha+k}$.
If $l<n_{\alpha+k}-1$ then we know that $e_{n+1}=u(e_n)$ and $v(e_n)=0$, whence $w(e_n)=e_{n+1}$.
Assume that $l=n_{\alpha+k}-1$. Then, $u(e_n) \in V_{\alpha+k}$ whereas $v(e_n)-e_{n+1} \in V_{\alpha+k}$,
and hence $w(e_n)-e_{n+1} \in V_{\alpha+k}=\underset{\beta \in \kappa, \; \beta<\alpha}{\sum} V_\beta+\Vect(e_0,\dots,e_n)$.

Hence, Lemma \ref{elementarylemma} applies to the endomorphism $\overline{w}$ of $E$ induced by $w$,
and it yields that the resulting module $E^{\overline{w}}$ is non-zero and free with generator $y$ (the class of $x_\alpha$ in $E$).
Therefore, $(W_\alpha)_{\alpha \in D}$ is a stratification of $V^w$ and all the terms of its dimension sequence
equal $+\infty$. By Lemma \ref{allinfinity}, we conclude that $w$ is elementary.
\end{proof}

\section{Decompositions into four quadratic operators}\label{theo4proofsection}

Here, we shall prove Theorem \ref{theo4}. Combining Remark \ref{canonicalremark} with Theorem \ref{theo2},
one sees that we only need to prove the following result.

\begin{prop}\label{reductiontoelementary}
Let $u$ be an endomorphism of an infinite-dimensional vector space $V$, and let $a$ and $b$ be scalars.
Then, there exist endomorphisms $u_1$ and $u_2$ of $V$ such that $u-u_1-u_2$ is elementary,
$u_1^2=a\,u_1$ and $u_2^2=b\,u_2$.
\end{prop}

We will prove Proposition \ref{reductiontoelementary} by
constructing a ``well-behaved" stratification of $V^u$.
In the first section, we perform such a construction, and in the
subsequent one we use it to construct the claimed endomorphisms $u_1$ and $u_2$ by way of a connector.

\subsection{On the existence of a well-behaved stratification}

\begin{prop}\label{existstrat}
Let $V$ be an $\F[t]$-module with infinite dimension as a vector space over $\F$.
Then, there is a stratification $(V_\alpha)_{\alpha \in \kappa}$ of $V$ such that:
\begin{enumerate}[(a)]
\item $\kappa$ is a cardinal;
\item If $\kappa$ is finite then $\dim_\F V_0=+\infty$.
\end{enumerate}
\end{prop}

\begin{proof}
We construct such a stratification by transfinite induction.
First we denote by $\nu$ the dimension of the $\F$-vector space $V$, seen as a cardinal.
We choose a basis $(e_k)_{k \in \nu}$ of $V$.
If $V$ is not a torsion module, we can further assume that $e_0$ does not belong to the torsion submodule of
$V$, i.e.\ that $\F[t]\,e_0$ is a free submodule of $V$.
Then, we construct an ordinal $\kappa \leq \nu$ and an increasing sequence $(V_\alpha)_{\alpha \in \kappa}$
of submodules of $V$ as follows.

We put $V_0:=\F[t]\,e_0$.

Let $\alpha \in \nu$ and assume that we have constructed an increasing sequence $(V_\beta)_{\beta<\alpha}$
of submodules of $V$ such that, for all $\beta<\alpha$:
\begin{enumerate}[(i)]
\item The $\F[t]$-module $V_\beta/\underset{\gamma<\beta}{\sum} V_\gamma$ is non-zero and has a generator;
\item The vector $e_\beta$ belongs to $V_\beta$.
\end{enumerate}
Put $W:=\underset{\beta<\alpha}{\sum} V_\beta$.
If $V=W$, then the process terminates at $\alpha$.
Otherwise, we take the least $k \in \nu$ such that $e_k \not\in W$ (note that $\alpha \leq k$), and we put
$V_\alpha:=W+\F[t] e_k$. The module $V_\alpha/W$ is non-zero and has a generator (namely, the class of $e_k$).
Finally, we note that $e_{\alpha} \in V_\alpha$. Hence, the inductive step is climbed.

Assume first that the process never terminates. Then, by condition (ii),
we have $V=\underset{\alpha \in \nu}{\sum} V_\alpha$.
It follows that $(V_\alpha)_{\alpha \in \nu}$ is a stratification of $V$, and $\nu$ is an infinite cardinal.

In the rest of the proof, we assume that the process terminates at some ordinal $\alpha \in \nu$,
so that $(V_\beta)_{\beta<\alpha}$ is a stratification of $V$.
First of all, we prove that $\nu$ is countable:
Let us take an associated vector sequence $(x_\beta)_{\beta <\alpha}$. Then, we know that
$V=\underset{\beta<\alpha}{\sum} \F[t] x_\beta$, and hence, denoting by $\mu$ the cardinality of $\alpha$, we find
$$\dim_\F V \leq \mu \times \aleph_0.$$
Note that $\mu<\nu$.
If $\aleph_0<\nu$, then $\mu \times \aleph_0<\nu$ and we have a contradiction.
If follows that $\nu=\aleph_0$ and hence $\alpha$ is a finite ordinal.
Next, as $V$ is infinite-dimensional at least one of the $\F$-vector spaces $\F[t] x_\beta$, for $\beta<\alpha$,
must be infinite-dimensional. Hence, $V^u$ is not a torsion module and by our assumptions $\dim_\F V_0=\dim_\F \F[t]e_0=+\infty$.
Hence, the claimed conclusion follows.
\end{proof}

\subsection{The reduction to an elementary endomorphism}

We are now ready to prove Theorem \ref{theo4}.
Here is the key step:

\begin{prop}\label{from4to2}
Let $u$ be an endomorphism of a vector space $V$.
Let $(V_\alpha)_{\alpha \in \kappa}$ be a stratification of $V^u$ that is indexed over an ordinal $\kappa$.
Assume that $\kappa$ is a limit ordinal or $V_0$ is infinite-dimensional.

Let $a$ and $b$ be scalars. Then, there exist endomorphisms $u_1$ and $u_2$ of $V$
such that $u_1^2=a\,u_1$, $u_2^2=b\, u_2$ and $u-(u_1+u_2)$ is elementary.
\end{prop}

The combination of this result with Proposition \ref{existstrat} yields Proposition \ref{reductiontoelementary},
which in turn yields Theorem \ref{theo4}.

\begin{proof}
We take the dimension sequence $(n_\alpha)_{\alpha \in \kappa}$ of $V^u$ and an associated vector sequence
$(x_\alpha)_{\alpha \in \kappa}$.
Given an ordinal $\alpha$, there is a largest $k\in \N$ for which there exists an ordinal
$\gamma$ satisfying $\alpha=\gamma+k$: we say that $\alpha$ is \textbf{even} when $k$ is even,
and \textbf{odd} otherwise.

\vskip 2mm
\noindent \textbf{Case 1: $\kappa$ is a limit ordinal.} \\
We define two endomorphisms $u_1$ and $u_2$ of $V$ as follows on the basis
$(u^k(x_\alpha))_{\alpha \in \kappa, 0 \leq k<n_\alpha}$:

For each ordinal $\alpha \in \kappa$ such that $n_\alpha<+\infty$, we set
$$u_1\bigl(u^{n_\alpha-1}(x_\alpha)\bigr):=\begin{cases}
a  \,u^{n_\alpha-1}(x_\alpha)- x_{\alpha+1} & \text{if $\alpha$ is even} \\
0 & \text{if $\alpha$ is odd}
\end{cases}$$
$$u_2\bigl(u^{n_\alpha-1}(x_\alpha)\bigr):=\begin{cases}
0 & \text{if $\alpha$ is even} \\
b  \,u^{n_\alpha-1}(x_\alpha)- x_{\alpha+1} & \text{if $\alpha$ is odd}
\end{cases}$$
and $u_1$ and $u_2$ are required to map all the other basis vectors to $0$.
We check on the above basis that $u_1^2=a \,u_1$ and $u_2^2=b \,u_2$.
Indeed, let $\alpha$ be an even ordinal such that $n_\alpha<+\infty$.
Then, $u_1(u^{n_\alpha-1}(x_\alpha))=a \, u^{n_\alpha-1}(x_\alpha)- x_{\alpha+1}$ and $\alpha+1$
is odd whence $u_1(x_{\alpha+1})=0$, which leads to $u_1^2(u^{n_\alpha-1}(x_\alpha))=a \, u_1(u^{n_\alpha-1}(x_\alpha))$.
For any other basis vector $y$, we have $u_1(y)=0$ and hence $u_1^2(y)=0=a\, u_1(y)$.
The proof is similar for $u_2$.

Next, it is obvious that $v:=-u_1-u_2$ is a connector for $u$ with respect to the vector sequence $(x_\alpha)_{\alpha \in \kappa}$.
We deduce from Proposition \ref{connector} that $u-u_1-u_2$ is elementary.

 \vskip 2mm
\noindent \textbf{Case 2: $\kappa$ has a maximum $M$.} \\
Then, we know that $V_0$ is infinite-dimensional. If $M=0$, then $u$ is already elementary and we just take $u_1=u_2=0$. In the rest of the proof, we assume that $M>0$.
Set $W:=\Vect(u^k(x_\alpha))_{0<\alpha \leq M, \; 0 \leq k<n_\alpha}$
and note that $V=V_0\oplus W$. Denote by $\pi$ the projection from $V$ onto $V_0$ along $W$, and set $w:=\pi \circ u \circ (\id-\pi)$.
Then, $u':=u-w$ stabilizes both $V_0$ and $W$.
For $\alpha \in \kappa \setminus \{0\}$, set $W_\alpha:=\Vect(u^k(x_\beta))_{0<\beta \leq \alpha, \; 0 \leq k<n_\beta}$
and note that $u'$ stabilizes $W_\alpha$. Set finally $W_0:=V$.
Now, define $D:=\kappa$ with the same ordering on $\kappa \setminus \{0\}$, but with $\alpha<0$ for all $\alpha \in \kappa \setminus \{0\}$.
One sees that $D$ is still a well-ordered set and now $(W_k)_{k \in D}$ is a stratification
for $V^{u'}$ for which the last dimension equals $+\infty$. An associated vector sequence is $(x_k)_{k \in D}$.

Just like in Case 1, we define a connector $v$ of $u'$ for this stratification and this vector sequence.
Hence, $u-(w-v)$ is elementary.
In order to conclude, it only remains to split $(w-v)$ into $u_1+u_2$ where $u_1^2=a\, u_1$ and $u_2^2=b\, u_2$.

First of all, we see that $(w-v)(u^k(x_0))=0$ for all $k \in \N$.

For each $\alpha \in \kappa \setminus \{0\}$ and each integer $k$ such that $0 \leq k <n_\alpha$, either
$k \neq n_\alpha-1$ and hence $(w-v)(u^k(x_\alpha))=0$, or $k=n_\alpha-1$ and hence
there is a vector $y_\alpha$ of $V_0$ such that
$$(w-v)(u^k(x_\alpha))=\begin{cases}
a\, u^k(x_\alpha)-x_{\alpha+1}+y_\alpha & \text{if $\alpha$ is odd} \\
b\, u^k(x_\alpha)-x_{\alpha+1}+y_\alpha & \text{if $\alpha$ is even.}
\end{cases}$$
Now, we define endomorphisms $u_1$ and $u_2$ as follows on the basis $(u^k(x_\alpha))_{\alpha \in \kappa, 0 \leq k<n_\alpha}$:
For all $\alpha \in \kappa$ such that $n_\alpha<+\infty$,
$$u_1\bigl(u^{n_\alpha-1}(x_\alpha)\bigr)=\begin{cases}
a\,u^{n_\alpha-1}(x_\alpha)-x_{\alpha+1}+y_\alpha & \text{if $\alpha$ is odd} \\
0 & \text{if $\alpha$ is even,}
\end{cases}$$
and
$$u_2\bigl(u^{n_\alpha-1}(x_\alpha)\bigr)=\begin{cases}
0 & \text{if $\alpha$ is odd} \\
b\, u^{n_\alpha-1}(x_\alpha)-x_{\alpha+1}+y_\alpha & \text{if $\alpha$ is even,}
\end{cases}$$
and $u_1$ and $u_2$ are required to map all the other basis vectors to $0$.
Obviously $u_1+u_2=w-v$. On the other hand, as in Case 1 it is easily checked that
$u_1^2=a\, u_1$ and $u_2^2=b\, u_2$ by using the fact that $u_1$ and $u_2$ vanish everywhere on $V_0$
(and in particular they map any $y_\alpha$ to $0$).
Since $u-(u_1+u_2)$ is elementary, the proof is complete.
\end{proof}

Therefore, Theorem \ref{theo4} is now fully established.

\end{document}